\documentclass{amsart}[12pt]

\usepackage{color}
\newtheorem{theorem}{Theorem}[section]
\newtheorem{lemma}[theorem]{Lemma}

\theoremstyle{definition}
\newtheorem{definition}[theorem]{Definition}
\newtheorem{example}[theorem]{Example}

\newtheorem{hypothesis}[theorem]{Hypothesis}

\theoremstyle{remark}
\newtheorem{remark}[theorem]{Remark}

\numberwithin{equation}{section}

\newcommand{\bbR}{\mathbb R}

\newcommand{\bbD}{\mathbb D}
\newcommand{\bbC}{\mathbb C}

\renewcommand{\epsilon}{\varepsilon}

\newcommand{\be}{\begin{equation}}
\newcommand{\ee}{\end{equation}}

\newcommand{\cH}{{\mathcal H}}

\newcommand{\cL}{{\mathcal L}}

\renewcommand{\Im}{{\ensuremath{\mathrm{Im}}}}

\newcommand{\fC}{\mathfrak{C}}
\newcommand{\fD}{\mathfrak{D}}

\newcommand{\fM}{\mathfrak{M}}

\newcommand{\fS}{\mathfrak{S}}

\DeclareMathOperator{\Ker}{\mathrm{Ker}}

\newcommand{\linspan}{\mathrm{lin\ span}}

\newcommand{\Dom}{\mathrm{Dom}}
\newcommand{\dom}{\mathrm{Dom}}

\date{\today}

\begin{document}

\title [ The addition and  multiplication theorems]
{On the addition and  multiplication theorems}

\author{K. A. Makarov}
\address{Department of Mathematics, University of Missouri, Columbia, MO 63211, USA}
\email{makarovk@missouri.edu}

\author{E. Tsekanovski\u{i} }
\address{
 Department of Mathematics, Niagara University, P.O. Box 2044,
NY  14109, USA }
\email{tsekanov@niagara.edu}

\dedicatory
{Dedicated with great pleasure to Lev Aronovich Sakhnovich
 on the occasion of his 80th birthday anniversary}



\begin{abstract} We discuss the classes $\fC$, $\fM$, and $\fS$ of analytic functions that can be realized as the    Liv\v{s}ic characteristic functions of a symmetric densely defined operator $\dot A$ with deficiency indices $(1,1)$, the Weyl-Titchmarsh functions  associated with the pair $(\dot A, A)$ where $A$ is a self-adjoint extension of $\dot A$, and the characteristic function of a maximal dissipative extension  $\widehat A$ of $\dot A$,  respectively. We show that the class $\fM$ is a convex set, both of the classes $\fS$  and $\fC$ are
  closed under multiplication and, moreover,   $\fC\subset \fS$ is a double sided ideal in the sense that $\fS\cdot \fC=\fC\cdot \fS\subset \fS$.
  The goal of this paper is to obtain  these analytic  results  by providing explicit constructions for the corresponding operator realizations.
 In particular, we introduce the concept of an operator coupling  of two unbounded
 maximal dissipative operators and establish an analog
of the  Liv\v{s}ic-Potapov  multiplication theorem \cite{LP} for the operators associated with the function  classes $\fC$
 and $\fS$. We also establish that the modulus  of the von Neumann parameter characterizing the domain of $\widehat A$ is a multiplicative functional with respect to the operator coupling.

\end{abstract}

 \maketitle

 \section{introduction }

In 1946, M. Liv\v{s}ic \cite{L}
introduced fundamental concepts of the  {\it characteristic functions} of a densely defined symmetric operator $\dot A$ with deficiency indices $(1, 1)$,  and of  its maximal non-self-adjoint  extension $\widehat A$. Under the hypothesis that the symmetric operator $\dot A$ is prime\footnote{Recall that a closed symmetric operator $\dot A$ is called a prime operator if $\dot A$ does not have invariant subspaces where the corresponding restriction of $\dot A$ is  self-adjoint},
a cornerstone result \cite[Theorem 13]{L} (also see \cite{AkG} and \cite{AWT})  states   that  the characteristic function  (modulo inessential constant unimodular factor) determines the operator  up to unitary equivalence.
In 1965,  in an attempt  to characterize  self-adjoint extensions  $A$ of a symmetric operator $\dot A$,
    Donoghue \cite{D} introduced the Weyl-Titchmarsh function associated with the pair $(\dot A, A)$ and  showed   that
   the  Weyl-Titchmarsh function determines the pair $(\dot A, A)$ up to unitary equivalence whenever  $\dot A$ is a prime symmetric operator  with deficiency indices $(1,1)$.

  In our recent  paper \cite{MT},  we introduced into  play an auxiliary self-adjoint (reference)
 extension $A$  of $\dot A$ and
     suggested to define
 the characteristic functions  of a symmetric operator and of its dissipative extension as the functions associated with  the pairs $(\dot A, A)$  and  $(\widehat A, A)$, rather than with the single operators $\dot A$ and $\widehat A$, respectively. Honoring  M. Liv\v{s}ic's fundamental contributions to the theory of non-self-adjoint operators and also taking into account the crucial role that  the characteristic function of  a symmetric operator plays in the theory, we   suggested to
  call the characteristic function associated with the pair $(\dot A, A)$ the Liv\v{s}ic function.
 For a detailed treatment  of the aforementioned  concepts of the   Liv\v{s}ic,   Weyl-Titchmarsh, and the characteristic functions   including the discussion of  their interrelations  we refer to \cite{MT}.

The main goal of this paper is to obtain the following two principal results.

Our first  result states  that  given two
 Weyl-Titchmarsh functions  $M_1=M(\dot A_1, A_1)$ and  $M_2=M(\dot A_2, A_2)$,
any convex combination
$pM_1+qM_2$ can also be realized as the Weyl-Titchmarsh function associated with a pair $(\dot A, A_1\oplus A_2)$, where
 $\dot A$ stands for some special
symmetric extension with deficiency indices $(1,1)$ of the direct orthogonal  sum  of $\dot A_1$ and $\dot A_2$
(see Theorem \ref{additionth}).

Our second  result  concerns the computation of the characteristic function
of an operator coupling $\widehat A=\widehat A_1\uplus \widehat A_2$ of two dissipative operators $\widehat A_1$ and $ \widehat A_2$, acting in the Hilbert spaces $\cH_1$ and $\cH_2$,  defined as
a dissipative extension of $\widehat A_1$ outgoing from the Hilbert space $\cH_1$ to the direct sum of the  Hilbert space $\cH_1\oplus \cH_2$  satisfying the constraint
$$
\widehat A|_{\Dom(\widehat A)\cap \dom((\widehat A)^*)}\subset \widehat A_1\oplus (\widehat A_2)^*.
$$
This result, called the multiplication theorem (see Theorem \ref{opcoup}),
 states that   the product $S_1\cdot S_2$  of   the characteristic  functions $S_1$ and $S_2$ associated with the pairs $(\widehat A_1, A_1)$ and
 $(\widehat A_2, A_2)$   coincides with the
characteristic function of  the operator coupling $\widehat A=\widehat A_1\uplus \widehat A_2$ relative to an appropriate reference self-adjoint operator.

It is important to mention that the multiplication theorem   substantially relies on the multiplicativity of   the absolute value  $\widehat \kappa (\cdot)$  of the von Neumann extension parameter
 of a maximal dissipative extension  of $\dot A$  established in  Theorem \ref{nachalo}:
\begin{equation}\label{vnkappa}
\widehat \kappa(\widehat A_1\uplus \widehat A_2)=\widehat \kappa(\widehat A_1)\cdot \widehat \kappa(\widehat A_2).
\end{equation}

Introducing the analytic function  classes  $\fC$ and $\fM$, elements of which
 can be realized as  the Liv\v{s}ic and   Weyl-Titchmarsh functions    associated with a pair
$(\dot A, A)$, respectively,
 along with the analytic function class $\fS$  consisting of all  characteristic functions associated with all possible  pairs
$(\widehat A, A)$,
as a corollary of our geometric considerations we obtain that

\begin{itemize}
\item[(i)] The class $\fM$  is a convex set with respect to addition;
\item[(ii)] The class $\fS$  is closed with respect to  multiplication,
$\fS\cdot \fS\subset \fS;
$
\item[(iii)] The subclass  $\fC\subset \fS$   is  a (double sided) ideal\footnote{We borrow this term from the ring theory. However, it worth perhaps mentioning that  the function class $\fS$ as an
algebraic structure is not a ring.} in the sense that
$$
\fC\cdot \fS=\fS\cdot \fC\subset \fC;
$$
\item[(iv)] The class $\fC$   is closed with respect to  multiplication:
$\fC\cdot\fC\subset \fC.$

\end{itemize}

The closedness of the class $\fS$ under multiplication (ii)  is a scalar variant of the multiplication theorem in the unbounded setting.
The multiplication theorem for bounded operators was originally obtained   in 1950 by
 M.~S.~Liv\v{s}ic and V.~ P.~Potapov \cite{LP}, who in particular
  established that the
product of two characteristic matrix-valued functions of bounded operators coincides with   the
matrix-valued  characteristic function of  a bounded
operator. After this  the result  has been  extended
 to the case of operator colligations (systems) \cite {ABT}, \cite{Br}, \cite {Brod}, \cite{BL58}, \cite{L2}, \cite{LYa}.

 The paper is organized as follows.

In Section 2, we recall the definitions and briefly discuss various  properties of the  Liv\v{s}ic, Weyl-Titchmarch and the characteristic functions.

 In Section 3, we introduce a coupling  of two symmetric operators defined as a
  symmetric extension   with deficiency indices $(1,1)$  of the direct sum of two symmetric operators $\dot A_1$ and $\dot A_2$ acting in the Hilbert spaces $\cH_1$ and $\cH_2$
  and then we explicitly compute the Liv\v{s}ic function of  the coupling (see Theorem \ref{technich0}).

 In Section 4, we prove  the {\it Addition Theorem for the Weyl-Titchmarsh functions} (see Theorem \ref{additionth}).

 In Section 5,  we develop  a variant of the extension theory with constrains, introduce a concept of the  operator coupling of  two unbounded  dissipative operators, discuss its properties,
 and prove  the {\it  Multiplicativity  of  the von Neumann extension parameter} (see Theorem \ref{nachalo}).

 In Section 6, we prove the   {\it Multiplication Theorem for the characteristic functions}  (see Theorem \ref{opcoup}).
 We also illustrate the corresponding  geometric constructions by  an  example of the differentiation operator on a finite interval
(see  Example \ref{exam}).

 In Appendix $A$,  a differentiation operator on a finite  interval is treated in detail (also see \cite{AkG} for  a related exposition).

\section{Preliminaries}
Throughout this paper we assume the following hypothesis.

\begin{hypothesis}\label{hyp}
Suppose that $\dot A$ is a densely defined symmetric operator $\dot A$ with deficiency indices $(1,1)$  and $A$ its self-adjoint extension.
Assume that  the deficiency elements  $g_\pm\in \Ker( (\dot A)^*\mp iI)$ are chosen in such a way that $\|g_+\|=\|g_-\|=1$
and that
\begin{equation}\label{ddff}
 g_{+}-g_-\in \Dom(A).
  \end{equation}
\end{hypothesis}
  \subsection{The Liv\v{s}ic function and the class $\fC$}
  Under Hypothesis \ref{hyp},  introduce    the  Liv\v{s}ic function $s=s(\dot A, A)$ of the symmetric operator $\dot A$ relative to  the self-adjoint extension $A$ by
 \begin{equation}\label{charfunk}
s(\dot A, A )(z)=\frac{z-i}{z+i}\cdot \frac{(g_z, g_-)}{(g_z, g_+)}, \quad z\in \bbC_+,
\end{equation}
where $g_z$, $z\in \bbC_+$,  is   an arbitrary deficiency element,  $0\ne g_z\in \Ker((\dot A)^*-zI )$.

We remark that  from the definition it follows that the dependence of
  the Liv\v{s}ic function
   $s(\dot A, A)$ on the reference (self-adjoint) operator
  $A$
 reduces to  multiplication by a $z$-independent
 unimodular factor whenever $A$ changes. That is,
 \begin{equation}\label{transs}
 s(\dot A, A_\alpha)=e^{-2i\alpha}s(\dot A, A), \quad \alpha \in [0,\pi),
 \end{equation}whenever the  self-adjoint  reference extension $A_\alpha$ of $\dot A$ has the property
\begin{equation}\label{dom1}
g_+-e^{2i\alpha}g_-\in \Dom (A_\alpha).
\end{equation}

Denote by $\fC$
the class of all  analytic mappings from $\bbC_+$ into the unit disk $\bbD$ that can be realized as the  Liv\v{s}ic function
associated with  some pair $(\dot A, A)$.

The class $\fC$ can be   characterized as follows (see \cite{L}). An
 analytic mapping $s$ from the upper-half plane into the unit disk
belongs to the class $\fC$, $s\in \fC$, if and only if
\begin{equation}\label{vsea0}
s(i)=0\quad \text{and}\quad \lim_{z\to \infty}
z(s(z)-e^{2i\alpha})=\infty \quad \text{for all} \quad  \alpha\in
[0, \pi),
\end{equation}
$$
0< \varepsilon \le \text{arg} (z)\le \pi -\varepsilon.
$$

\subsection{The Weyl-Titchmarsh function and the class $\fM$}

Define
 the Weyl-Titchmarsh function $M(\dot A, A)$
associated with the pair $(\dot A, A )$  as
\begin{equation}\label{1}
M(\dot A, A)(z)=
\left ((Az+I)(A-zI)^{-1}g_+,g_+\right ), \quad z\in \bbC_+.
\end{equation}

Denote by $\fM$ the class of all analytic mapping from $\bbC_+$ into itself  that can be realized  as the Weyl-Titchmarsh function $M(\dot A, A)$ associated with a pair $(\dot A, A)$.

 As for the characterization of the class $\fM$, we recall  that $M\in \fM$ if and only if $M$ admits the representation (see  \cite{D}, \cite{FKE}, \cite{GT},
 \cite {MT})
 \begin{equation}\label{hernev0}
M(z)=\int_\bbR \left
(\frac{1}{\lambda-z}-\frac{\lambda}{1+\lambda^2}\right )
d\mu,
\end{equation}
where $\mu$ is an  infinite Borel measure   and
 \begin{equation}\label{hernev01}
\int_\bbR\frac{d\mu(\lambda)}{1+\lambda^2}=1\,,\quad\text{equivalently,}\quad M(i)=i.
\end{equation}

It is worth mentioning  (see, e.g., \cite{MT}) that the Liv\v{s}ic and Weyl-Titchmarsh functions are related by the Cayley transform
\begin{equation}\label{blog}
s(\dot A, A)(z)=\frac{M(\dot A, A)(z)-i}{M(\dot A,A)(z)+i},\quad z\in \bbC_+.
\end{equation}

Taking this into account, one can show  that  the  properties \eqref{vsea0} and \eqref{hernev0},
\eqref{hernev01} imply  one another  (see, e.g.,  \cite{MT}).

 Combining  \eqref{transs}, \eqref{dom1}  and \eqref{blog} shows  that
the corresponding transformation law for the Weyl-Titchmarsh functions reads as (see  \cite{D}, \cite{FKE}, \cite{GT})
\begin{equation}\label{transm}
M(\dot A, A_\alpha)=\frac{\cos \alpha \, M(\dot A, A)-\sin \alpha}{
\cos\alpha +\sin \alpha \,M(\dot A, A)},
\quad \alpha \in [0,\pi).
\end{equation}

In view of \eqref{blog}, the function classes $\fC$ and $\fM$ are related by the Cayley transform,
$$\fC=K\circ\fM,
$$
where
$$
K(z)=\frac{z-i}{z+i}, \quad z\in \bbC.
$$That is,
$$\fC=\{K\circ M\, |\, M\in \fM\},
$$ where $ K\circ M$ denotes  the composition of the functions $K$  and $M$.

Moreover, the transformation law \eqref{transs}   shows that the class $\fC$ is closed under multiplication by a unimodular constant,
\begin{equation}\label{theta}\theta \cdot \fC = \fC,\quad |\theta|=1.
\end{equation}
Accordingly,  from \eqref{transm} one concludes that  the class $\fM$ is closed under the action of a one parameter
subgroup  of $SL(2, \bbR)$ of linear-fractional transformations
$$
K_\alpha\circ \fM=\fM,
\quad \bbR\ni  \alpha  \mapsto   K_\alpha,$$
 given by  $$
K _\alpha(z)= \frac{\cos \alpha \, z-\sin \alpha}{
\cos\alpha +\sin \alpha \,z} .
$$

  \subsection{The von Neumann extension parameter of a dissipative operator}

 Denote by  $\fD$
the set of all maximal  dissipative unbounded operators $\widehat A$  such that the restriction  $\dot A$ of $\widehat A$ onto $ \Dom(\widehat A)\cap\Dom((\widehat A)^*)$ is a densely defined symmetric operators with indices $(1,1)$.

Given $\widehat A\in \fD$ and  a self-adjoint (reference) extension $A$ of the underlying symmetric operator $\dot A=\widehat A|_{\Dom(\widehat A)\cap\Dom((\widehat A)^*)}$, assume that the pair $(\dot A, A)$ satisfies Hypothesis \ref{hyp}
with some $g_\pm$ taken from the corresponding deficiency subspaces,
 so that
$$
g_+-g_-\in \Dom ( A).$$

In this case,
 \begin{equation}\label{unkappa}
g_+-\kappa g_-\in \Dom ( \widehat A)
\quad \text{for some } \quad \kappa \in \bbD.
\end{equation}

\begin{definition}
We call $\kappa=\kappa(\widehat A, A)$ the von Neumanmn extension parameter of the dissipative operator $\widehat A\in \fD$ relative to the reference self-adjoint operator $A$.
\end{definition}

\subsection{ The characteristic function of a dissipative operator  and the class $\fS$}
Suppose that $\widehat A\in \fD$ is a maximal dissipative operator,  $\dot A=\widehat A|_{\Dom(\widehat A)\cap\Dom((\widehat A)^*)}$ its symmetric restriction, and $A$ is a reference self-adjoint extension of $\dot A$.
Following \cite{L} (also see \cite{AkG}, \cite{MT}) we define
  the characteristic
function $S=S(\widehat A, A)$ of the dissipative operator $\widehat A$ relative to the reference self-adjoint operator $A$   as
\begin{equation}\label{ch1}
S(z)=\frac{s(z)-\kappa}
{\overline{ \kappa }\,s(z)-1},\quad z\in \bbC_+,
\end{equation}
where  $s=s(\dot A, A)$ is  the Liv\v{s}ic function  associated with the pair $(\dot A, A)$ and the complex  number $\kappa=\kappa(\widehat A, A)$ is the von Neumann extension parameter of $\widehat A$ (relative to $A$).

We stress that for a dissipative operator $\widehat A\in \fD$ one always has that
\begin{equation}\label{dota}
\Dom(\widehat A)\ne \Dom((\widehat A)^*)
\end{equation}
and, moreover,  the underlying densely defined symmetric  operator $\dot A$ can uniquely  be recovered by  restricting  $\widehat A$
on $$\Dom (\dot A)=\Dom (\widehat A)\cap \Dom\left ( ( \widehat A)^*\right  ).$$
This  explains why  it is more natural to associate  the characteristic function with the pair
$(\widehat A, A)$ rather than with   the triple $(\dot A, \widehat A, A)$ which would  perhaps be more pedantic.

  The class of all analytic mapping from $\bbC_+$ into the unit disk consisting of all the characteristic functions
  $S(\widehat A, A)$  associated with arbitrary  pairs $ (\widehat A, A)$,  with  $\widehat A\in \fD$ and $A$  a reference self-adjoint extension of  the underlining symmetric operator  $\dot A$, will be denoted by $\fS$.

As in the case of the class $\fC$, the class $\fS$  is also closed under multiplication by a constant unimodular factor (cf. \eqref{theta}), that is,
  $$
  \theta \cdot \fS=\fS,\quad |\theta|=1.
  $$

 Indeed,  if $S\in \fS$,  then
    $$
  S=\frac{s-\kappa}
{\overline{ \kappa }\,s-1} \quad \text{for some } s\in \fC, \,\,\kappa\in \bbD.
  $$
Therefore,  $$
  \theta \cdot S=\frac{\theta \cdot  s- \theta \kappa}
{\overline{ \theta \kappa }\,\theta \cdot s-1}, \quad |\theta|=1.
  $$
  Since the class $\fC$ is closed under multiplication by a constant unimodular factor, $\theta\cdot  s\in \fC$ and since $|\theta \kappa|<1$,
  by definition \eqref{ch1}, the function
  $\theta \cdot S$
   belongs to the class $\fS$ as well.

     Now it is easy to see that the class $\fS$ coincides with the orbit of the class $\fC$ under the action
  of the group   of  automorphisms  $\text{Aut} (\bbD)$ of the complex unit disk. That is,
  \begin{equation}\label{kcirc}
  K\circ \fC=\fS,
\quad K\in\text{Aut} (\bbD).
  \end{equation}

  In particular, one obtains that
   \begin{equation}\label{subsub}
  \fC\subset \fS.
  \end{equation}

From \eqref{kcirc} and \eqref{subsub} follows that the class $\fS$  is closed under the action of the
group  $\text{Aut} (\bbD)$, that is,
$$
K\circ \fS=\fS,
\quad K\in\text{Aut} (\bbD). $$

\subsection{The  unitary invariant $\widehat \kappa:\fD\to [0,1)$}\label{ssylka}  Combining  \eqref{vsea0} and \eqref{ch1} shows that the value  of the von Neumann extension  parameter  $\kappa (\widehat A, A)$ can also  be recognized  as the value of the characteristic function at the point $z=i$, that is,
 \begin{equation}\label{kappa1}\kappa=\kappa(\widehat A, A)=S(\widehat A, A)(i).
 \end{equation}
Since by  Liv\v{s}ic theorem \cite[Theorem 13]{L}   the characteristic function  $S(\widehat A, A)$ determines  the pair $(\widehat A, A)$ up to unitary equivalence provided that the underlining symmetric operator $\dot A$ is prime, cf. \cite{MT}, the parameter $\kappa$ is a unitary invariant of the pair $(\widehat A , A)$.

It is important to notice that
the absolute value $\widehat \kappa(\widehat A)=|\kappa(\widehat A, A)|$  of the von Neumann extension parameter
 is independent  of the choice of the reference self-adjoint extension $A$. Therefore,
 the following functional
  $$\widehat \kappa :\fD\to [0,1)
  $$
  of the form
  $$ \widehat \kappa=\widehat \kappa(\widehat A)=|\kappa(\widehat A, A)|, \quad  \widehat A\in \fD,
  $$  is well defined  as one of   the geometric  unitary  invariants of a dissipative operator from the class $\fD$.

The kernel of the functional $\widehat \kappa$ can be characterized as follows.

The  inclusion \eqref{subsub} shows that
  any Liv\v{s}ic function  $s(\dot A, A)$ can be indentified with  the characteristic function $S(\widehat A,A')$ associated with  some pair $(\widehat A, A')$ where  $\widehat A\in \fD$ and $A'$ is  an appropriate self-adjoint reference extension of the symmetric operator  $\dot A
  =\widehat A|_{\Dom(\widehat A)\cup \Dom ((\widehat A)^*)}$.

To be more specific, it suffices to take
   the maximal  dissipative extension $\widehat A$ of $\dot A$
    with the   domain
  \begin{equation}\label{vanish}
  \Dom(\widehat A)=\Dom(\dot A)\dot +\Ker ((\dot A)^*-iI)
  \end{equation}
and to choose   the reference self-adjoint operator
    $A'$ in such a way that
   $$
   s(\dot A, A)=-s(\dot A, A').
   $$
   This is always possible due to  \eqref{transs}. Since $\widehat \kappa(\widehat A)=0$ (combine \eqref{unkappa} and \eqref{vanish}), it is easy to see that
  $$
 S	(\widehat A, A')=-s(\dot A, A')=s(\dot A,A)
  $$
which proves the claim.

  The subclass of  maximal  dissipative extensions $\widehat A$ with the property  \eqref{vanish} will be denoted by $\dot \fD$. That is, 		
  \begin{equation}\label{dotD}
  \dot \fD=\{\widehat A\in \fD\,|\, \widehat \kappa(\widehat A)=0\}\subset \fD,
 \end{equation}
and, therefore,
 $$
 \dot  \fD=\Ker ( \widehat \kappa).
 $$

\section{Symmetric extensions of the direct sum of symmetric operators}

Suppose that  $\dot A_1$ and
$\dot A_2$ are densely defined  symmetric operators  with deficiency indices $(1,1)$ acting in the Hilbert spaces $\cH_1$ and $\cH_2, $
respectively.

In accordance with the von Neumann extensions theory,
the set of all symmetric extensions $\dot A$
 with deficiency indices $(1,1)$
of the direct sum of the  symmetric operators $\dot A_1\oplus\dot A_2$
is in one-to-one correspondence with the set of
 one-dimensional {\it neutral}
subspaces $\cL$ of the quotient space
$\Dom ((\dot A_1\oplus\dot A_2)^*)\slash
\Dom (\dot A_1\oplus\dot A_2)$
such that  the adjoint operator
$(\dot A_1\oplus\dot A_2)^*$ restricted on $\cL$ is symmetric,
that is,
$$
\Im ((\dot A_1\oplus \dot A_2)^*f, f)=0, \quad f\in \cL.
$$
The above mentioned correspondence can be established in the following way: given $\cL$, the corresponding  symmetric operator
$
\dot A
$
is determined by the restriction of
$(\dot A_1\oplus \dot A_2)^*
$ on
 $$
\Dom(\dot A)=\Dom(\dot A_1)\oplus \Dom (\dot A_2)\dot +\cL,
$$ and vice versa.

Our  main technical result   describes
 the geometry of the deficiency subspaces of the
symmetric extensions $\dot A$ associated with  a two-parameter family of
neutral subspaces  $\cL$.  We also explicitly obtain the Liv\v{s}ic function of  these symmetric extensions  $\dot A$ relative to an
 appropriate self-adjoint extension of $\dot A$.

\begin{theorem}\label{technich0} Assume that  $\dot A_k$, $k=1,2$,
 are closed
symmetric operators with deficiency indices $(1,1)$ in the Hilbert spaces
$\cH_k$,  $k=1,2$.
Suppose that
$
g_k\in \Ker ((\dot A_k)^*\mp  iI)$, $   \|g_\pm^k\|=1$, $  k=1,2$.

Introduce   the  one-dimensional
 subspace $\cL\subset \cH_1\oplus \cH_2$  by
$$
\cL
=\linspan\left \{ (\sin \alpha g_+^1-\sin \beta  g_-^1)
\oplus(\cos \alpha
g_+^2- \cos\beta g_-^2 )
\right \},
$$
$$\alpha, \beta\in [0, \pi).
$$

Then \begin{itemize}
\item[(i)]
the linear set $\cL$ is a neutral subspace  of the quotient  space
$$\Dom ((\dot A_1\oplus\dot A_2)^*)\slash
\Dom (\dot A_1\oplus\dot A_2),$$
\item[(ii)] the  restriction $\dot A$ of
the operator  $(\dot A_1\oplus \dot A_2)^*$
on the domain
$$
\Dom(\dot A)=\Dom (\dot A_1)\oplus \Dom (\dot A_2)\dot +\cL
$$
is a symmetric operator with deficiency indices $(1,1)$
and the deficiency subspaces of $\dot A$  are given by
$$
\Ker ((\dot A)^*\mp iI)=\linspan \{ G_\pm \},
$$
where
\begin{equation}\label{G+G-}
G_+=  \cos  \alpha g_+^1-\sin  \alpha g_+^2
\,\, \text{ and } \,\,
G_-= \cos \beta  g_-^1- \sin \beta g^2_-,
\end{equation}
$$\|G_\pm\|=1;
$$

\item[(iii)] the Liv\v{s}ic function
$s=s(\dot A, A)$ associated with the pair  $(\dot A, A)$, where $A$ is a reference   self-adjoint extension of $\dot A$
such that
$$G_+-G_-\in \Dom(A),
$$
admits the representation
\begin{equation}\label{formula}
s(z)=\frac{\cos \alpha \cos \beta  s_1(z)- s_1(z)s_2(z)
+ \sin \alpha \sin \beta  s_2(z)
}
{1   - (\sin \alpha \sin \beta s_1(z)+
\cos  \alpha \cos  \beta s_2(z) ) },\quad z\in \bbC_+.
\end{equation}
Here
$s_k=s(\dot A_k, A_k)
$
are the Liv\v{s}ic functions associated with the pairs $(\dot A_k, A_k)$, $k=1,2.$
\end{itemize}

\end{theorem}


\begin{proof} (i).
First we note that the element $f\in \cL\subset \dom (\dot A)$ given by
\begin{equation}\label{stali1}
f=(\sin \alpha g_+^1-\sin \beta  g_-^1)
+(\cos \alpha
g_+^2- \cos\beta g_-^2 )
\end{equation}
belongs to $ \dom ((\dot A_1\oplus \dot A_2)^*)$
and that
\begin{equation}\label{stali2}
(\dot A_1\oplus \dot A_2)^*f=i(\sin \alpha g_+^1+\sin \beta  g_-^1
+\cos \alpha
g_+^2+\cos\beta g_-^2 ).
\end{equation}
Combining \eqref{stali1} and \eqref{stali2}, one obtains
\begin{align*}
((\dot A_1\oplus \dot A_2)^*f,f)=&i(\sin^2\alpha-\sin^2\beta+\cos^2\alpha-\cos^2\beta)
\\ \,\, &
+i\sin\alpha \sin \beta( (g^1_-, g^1_+) -(g^1_+, g^1_-))
\\ \,\, &+i
\cos \alpha \cos  \beta( (g^2_-, g^2_+) -(g^2_+, g^2_-)).
\end{align*}
Hence,
$$
\Im ((\dot A_1\oplus \dot A_2)^*f,f)=0
,\quad f\in \cL,$$
and therefore
$$
\Im ((\dot A_1\oplus \dot A_2)^*f,f)=0,\quad \text{ for all } \,\,
 f\in \dom(\dot A),
$$
which proves that   the operator $\dot A$ is symmetric and $(i)$ follows.

 (ii).  Let us show that
$$
\Ker ((\dot A)^*-iI)=\linspan \{ G_+ \}.
$$
We need to check that
$$
((\dot A+iI)y,G_+ )=0 \quad \text{for all}\quad  y\in\Dom(\dot A).
$$
Take a $y\in \Dom(\dot A)$. Then $y$ can be decomposed as
$$
y=h_1+h_2+Cf,
$$
where $h_k\in \Dom (\dot A_k)$, $k=1,2$,  $C\in \bbC$, and
\begin{equation}\label{fff}
f=(\sin \alpha g_+^1-\sin \beta  g_-^1)
\oplus(\cos \alpha
g_+^2- \cos\beta g_-^2 )\in \cL.
\end{equation}
Next,
\begin{align}
((\dot A+iI)y,G_+ )&=(\dot A+iI)(h_1+h_2+Cf), G_+)
\label{okm1}\\&=
((\dot A_1+iI)h_1\oplus(\dot A_2+iI)h_2),
 G_+)
\nonumber \\&+C((\dot A+iI)f,G_+).\nonumber
\end{align}
On the other hand, since $g_+^k\in \Ker ((\dot A_k)^*-iI)$, $k=1,2,$
\begin{align}
((\dot A_1+iI)h_1\oplus(\dot A_1+iI)h_2),
 G_+)&=
\cos  \alpha((\dot A_1+iI)h_1,g_+^1)
\label{okm2}\\&-\sin  \alpha
((\dot A_2+iI)h_2,g_+^2)=0.\nonumber
\end{align}
Now we can prove that
\begin{equation}\label{okm3}
((\dot A+iI)f,G_+)=0,\quad f\in \cL.
\end{equation}
Indeed,
\begin{align*}
((\dot A+iI)f&=((\dot A+iI)((\sin \alpha g_+^1-\sin \beta  g_-^1)
+(\cos \alpha
g_+^2- \cos\beta g_-^2 ))\\&=2i(\sin \alpha g_+^1
+\cos \alpha
g_+^2 )
\end{align*}
and since
\begin{equation}\label{g+g+}
G_+=  \cos  \alpha g_+^1-\sin  \alpha g_+^2,
\end{equation}
we have
 $$
((\dot A+iI)f,G_+)
=2i(\sin \alpha g_+^1
+\cos \alpha
g_+^2 ),\cos  \alpha g_+^1-\sin  \alpha g_+^2)=0.
$$
Combining \eqref{okm1}, \eqref{okm2} and \eqref{okm3}  proves that
$$
((\dot A+iI)y,G_+ )=0 \quad\text{for all} \quad y\in \Dom(\dot A).
$$
Therefore,
$$
G_+\in \Ker ((\dot A)^*-iI).
$$
In a similar way it follows  that $G_-$ given by
\begin{equation}\label{g-g-}G_-= \cos \beta  g_-^1- \sin \beta g^2_-
\end{equation}
generates the deficiency subspace $\Ker ((\dot A)^*+iI)$.

Since $\|g_\pm^{1}\|=\|g_\pm^2\|=1$ and the elements $g_\pm^1$ and $g_\pm^2$ are orthogonal to each other,
\eqref{g+g+}
and \eqref{g-g-} yield
 \begin{equation}\label{normm} \|G_\pm \|=1.
\end{equation}

(iii).  In order to  evaluate the Liv\v{s}ic function associated with  the pair $(\dot A, A)$,
choose  nontrivial elements $g_z^k\in \Ker ( (\dot A_k)^*-zI)$, $k=1,2$, $z\in \bbC_+$.

Suppose that for $z\in \bbC_+$ an element $G_z\ne 0$ belongs to the deficiency subspace $\Ker ( (\dot A)^*-zI)$.
Since  $\dot A \subset (\dot A_1\oplus \dot A_2)^*$, one gets that
$$
G_z=g_z^1+T(z)g_z^2\in \Ker ( (\dot A_1\oplus \dot A_2)^*-zI)
$$
for some function $T(z)$ (to be determined later).

Therefore, the Liv\v{s}ic function $s(z)=s(\dot A, A)(z)$ associated with the pair $(\dot A, A)$ admits the representation
\begin{align}
s(z)&=\frac{z-i}{z+i}\cdot \frac{(G_z, G_-)}{(G_z, G_+)}
=\frac{z-i}{z+i}\cdot \frac{(g_z^1+ T(z)g_z^2, \cos\beta  g_-^1- \sin \beta g^2_-)}
{(g_z^1+ T(z)g_z^2, \cos  \alpha g_+^1-\sin  \alpha g_+^2) }
\label{inis}\\
&=\frac{z-i}{z+i}\cdot
\frac{ \cos \beta (g^1_z, g^1_-)-T(z)\sin \beta  (g_z^2, g_-^2)
}
{ \cos  \alpha   (g^1_z, g^1_+)-T(z) \sin \alpha ( g_z^2, g_+^2)}.\nonumber
\end{align}

Since  $G_z\in \Ker ( (\dot A)^*-zI)$ implies that
\begin{equation}\label{G_Z}
(G_z, (\dot A-\overline{z}I)f)=0,
\end{equation}
where the element  $f\in \cL$ is given by \eqref{fff},
the equation \eqref{G_Z} yields  the following equation for determining  the function $T(z)$:
\begin{equation}\label{texn}
\left (g_z^1+T(z)g_z^2, (\dot A-\overline{z}I)\left (
(\sin \alpha g_+^1-\sin \beta g_-^1) \oplus
(\cos \alpha g_+^2- \cos \beta g_-^2)\right ) \right )=0.
\end{equation}
Since
$$\dot A \subset (\dot A_1\oplus \dot A_2)^*\quad \text{and}\quad
(\dot A_k)^* g^k_\pm =\pm i g^k_\pm, \quad k=1,2,
$$
from \eqref{texn} one gets that
\begin{align*}
&(-i-z) \sin \alpha (g_z^1, g_+^1)-(i-z)  \sin \beta(g_z^1, g_-^1)
\\
&+T(z)\left [(-i-z)\cos
\alpha (g_z^2, g_+^2) - (i-z)\cos \alpha (g_z^2, g_-^2)\right ]=0.
\end{align*}
Solving for $T(z)$, we have
\begin{align}
T(z)&=- \frac{(-i-z)\sin \alpha
(g_z^1, g_+^1)-(i-z)
 \sin \beta (g_z^1, g_-^1)}{(-i-z)
\cos \alpha (g_z^2, g_+^2) - (i-z)\cos \beta (g_z^2, g_-^2)}
\label{TTT}
\\
&
=-  \frac{(g_z^1, g_+^1)}{(g_z^2, g_+^2)}\cdot
\frac{
\sin \alpha -\sin \beta \frac{z-i}{z+i}\cdot
 \frac{ (g_z^1, g_-^1)}{(g_z^1, g_+^1)}}
{\cos \alpha
-\cos \beta\frac{z-i}{z+i}\cdot
 \frac{ (g_z^2, g_-^2)}{(g_z^2, g_+^2)}}
\nonumber \\&
=-  \frac{(g_z^1, g_+^1)}{(g_z^2, g_+^2)}\cdot
\frac{
\sin \alpha -\sin \beta s_1(z)}
{\cos \alpha
-\cos \beta  s_2(z)}. \nonumber
\end{align}

Therefore, taking into account \eqref{inis} and \eqref{TTT}, one arrives at the repsentation
\begin{align*}
s(z)&=\frac{z-i}{z+i}\cdot
\frac{ \cos \beta (g^1_z, g^1_-)+ \frac{(g_z^1, g_+^1)}{(g_z^2, g_+^2)}\cdot
\frac{
\sin \alpha -\sin \beta s_1(z)}
{\cos \alpha
-\cos \beta s_2(z)}\sin \beta  (g_z^2, g_-^2)
}
{ \cos  \alpha   (g^1_z, g^1_+)+ \frac{(g_z^1, g_+^1)}{(g_z^2, g_+^2)}\cdot
\frac{
\sin \alpha -\sin \beta s_1(z)}
{\cos \alpha
-\cos \beta s_2(z)}\sin \alpha ( g_z^2, g_+^2)}
\end{align*}
which, after a direct computation, yields \eqref{formula}.

The proof is complete.
\end{proof}

\begin{remark}\label{kogdage}  A straightforward computation using \eqref{formula}
shows that representation \eqref{formula} is a particular case (for $k=0$) of a more general  equality
  \begin{equation}\label{formula1}
\frac{s(z)-k}{
ks(z)-1}=\frac{a_1 s_1(z)+a_2 s_2(z)-
s_1(z)s_2(z)
-k}
{a_2 s_1(z)+a_1 s_2(z)-ks_1(z)s_2(z)-
1},\quad k\in [0,1).
\end{equation}
Here
$$
a_1=\cos \alpha \cos \beta +k\sin \alpha \sin \beta,
$$
$$
a_2=\sin \alpha \sin \beta+k\cos \alpha \cos \beta.
$$
\end{remark}

\section{The addition  theorem}

As the first application of Theorem \ref{technich0} we obtain the following addition theorem for the Weyl-Titchmarsh functions.

\begin{theorem}[{\bf The Addition Theorem}]\label{additionth}
Assume the hypotheses of Theorem \ref{technich0} with $\alpha=\beta.$
Suppose that $\dot A$ is the symmetric operator referred
to in Theorem
\ref{technich0}.

Then
the Weyl-Titchmarsh function $M$ associated with the pair
$(\dot A, A_1\oplus A_2)$
is  a convex combination
of  the Weyl-Titchmarsh functions  $M_k$
associated with the pairs $(\dot A_k, A_k)$, $k=1,2$,  which is given by
\begin{equation}\label{vypvyp}
M (z)=\cos^2 \alpha  \,\,M_1(z)+\sin^2 \alpha \,\,M_2(z),
\quad z\in \bbC_+.
\end{equation}

\end{theorem}


\begin{proof} Since by hypothesis $\alpha=\beta$,
one concludes that
$$
G_+-G_-\in \Dom (A_1\oplus A_2),
$$
where $G_\pm$ are the deficiency elements of $\dot A$ from
Theorem \ref{technich0} given by \eqref{G+G-}.
So, one can apply
 Theorem \ref{technich0} with the self-adjoint reference operator
$
A=A_1\oplus A_2
$ to conclude that
$$
s (z)=
\frac{\cos^2 \alpha  (z)-s_1(z)s_2(z)
+ \sin^2 \alpha \, s_2(z)
}
{1   - (\sin^2 \alpha  s_1(z)+
\cos^2  \alpha s_2(z) ) },
$$
where
$$
s(z)=\frac{M(z)-i}{M (z)+i}\quad \text{ and }\quad
 s_k=\frac{M_k(z)-i}{M_k(z)+i}, \quad k=1,2.
$$

Thus, to prove \eqref{vypvyp} it remains to check the equality
\begin{align*}
&\frac{\cos^2 \alpha  \,\,M_1(z)+\sin^2 \alpha \,\,M_2(z)-i}{\cos^2 \alpha  \,\,M_1(z)+\sin^2 \alpha \,\,M_2(z)+i}
\\
&=\frac{\cos^2 \alpha  \frac{M_1(z)-i}{M_1(z)+i}-
\frac{M_1(z)-i}{M_1(z)+i}\frac{M_2(z)-i}{M_2(z)+i}
+ \sin^2 \alpha \frac{M_2(z)-i}{M_2(z)+i}
}
{1   - \left (\sin^2 \alpha \frac{M_1(z)-i}{M_1(z)+i}+
\cos^2 \alpha  \frac{M_2(z)-i}{M_2(z)+i} \right ) }
\end{align*}
which can be directly verified.
\end{proof}

\section{An operator coupling of dissipative operators}

We now introduce the  concept of the operator coupling of two dissipative unbounded operators.

\begin{definition}

Suppose   that
  $\widehat A_1\in \fD(\cH_1)$   and $\widehat A_2\in \fD(\cH_2)$ are maximal dissipative unbounded operators acting in the Hilbert spaces $\cH_1$ and $\cH_2$, respectively.

We say that a maximal dissipative  operator  $\widehat A\in \fD(\cH_1\oplus\cH_2)$ is an operator coupling   of $\widehat A_1$ and $\widehat A_2$,
in writing, $$\widehat A=\widehat A_1\uplus \widehat A_2,$$   if
\begin{itemize}
\item[(i)] the Hilbert space $\cH_1$ is invariant for $\widehat A_1$ and
   the restriction of $\widehat  A$ on $\cH_1$ coincides with the dissipative operator $\widehat A_1$, that is,
   $$
   \Dom(\widehat A)\cap \cH_1=\Dom (\widehat A_1),
   $$
$$
\widehat A|_{\cH_1\cap \Dom({\widehat A_1)}}=\widehat A_1,
$$
\end{itemize}
and
\begin{itemize}
\item[(ii)]the symmetric operator  $\dot A=   \widehat A|_{\Dom(\widehat A)\cap \Dom ((\widehat A)^*)}$ has the property
$$
\dot A\subset \widehat A_1\oplus (\widehat  A_2)^*.
 $$
 \end{itemize}
\end{definition}

To  justify the existence of an operator coupling of two dissipative operators and discuss properties of the concept
we proceed with  preliminary considerations.

Assume the following hypothesis.
\begin{hypothesis}\label{hyphyp} Suppose   that
  $\widehat A_1\in \fD(\cH_1)$   and $\widehat A_2\in \fD(\cH_2)$ are maximal dissipative unbounded operators acting in the Hilbert spaces $\cH_1$ and $\cH_2$, respectively.
Assume, in addition, that
 $$
\dot A_j=\widehat A_j|_{\Dom(\widehat A_j)\cap\Dom((\widehat A_j)^*)},\quad j=1,2,
$$
are  the corresponding underlying symmetric operators.
\end{hypothesis}

First we show that under Hypothesis \ref{hyphyp} the following
 {\bf extension problem with a constraint} admits a one-parameter family of solutions.

 This problem is:
 \begin{itemize}
 \item[]
{\it Find a closed symmetric operator $\dot A$  with deficiency indices $(1,1)$ such that }
 \begin{equation}\label{EPC}
\dot A_1\oplus \dot A_2\subset \dot A
 \quad \text{and }\quad
\dot A\subset \widehat A_1\oplus (\widehat  A_2)^*
\end{equation}
\end{itemize}

The lemma below justifies the solvability of the extension problem with a constraint.

\begin{lemma}\label{dlinnaia}
Assume Hypothesis \ref{hyphyp}. Then

\begin{itemize}
\item[(i)] there exists a one parameter family $[0, 2\pi)\ni \theta\mapsto \dot A_\theta$ of symmetric restrictions   with deficiency indices $(1,1)$ of the operator  $(\dot A_1\oplus\dot A_2)^*$ such that
$$
\dot A_1\oplus \dot A_2\subset \dot A_\theta\subset \widehat A_1\oplus (\widehat  A_2)^*,\quad \theta\in [0,2\pi);
 $$

\item[(ii)] if $\dot A$ is a closed symmetric operator with deficiency indices $(1,1)$ such that
 $$
\dot A_1\oplus \dot A_2\subset \dot A\subset \widehat A_1\oplus (\widehat  A_2)^*,
 $$
 then there exists a $\theta\in [0, 2\pi)$ such that
 $$\dot A=\dot A_\theta.
 $$
 \end{itemize}
 \end{lemma}

\begin{proof} First, introduce the notation.
Let  $\kappa_j$,  $ 0\le \kappa_j<1$, $j=1,2, $ stand for the absolute value of the von Neumann parameter of $\widehat A_j$,
$$\kappa_j=\widehat \kappa (\widehat A_j), \quad j=1,2.
$$
Fix a basis
$g_\pm^j \in \Ker ((\dot A_j)^*\mp iI)$, $\|g_\pm^j\|=1$, $j=1,2$, in the corresponding deficiency subspaces
such that
$$
g_+^j-\kappa_jg_-^j\in \Dom (\widehat A_j),
\quad j=1,2.
$$

(i). To  show that there exists at least one symmetric extensions $\dot A_0$  with deficiency indices $(1,1)$ of  $\dot A_1\oplus\dot A_2$
such that
$$
\dot A_0\subset \widehat A_1\oplus (\widehat  A_2)^*,
 $$
suppose that $\alpha, \beta \in \big [0, \frac{\pi}{2} \big )$ are chosen in such a  way that
\begin{equation}\label{tana1}
\alpha=\begin{cases}
\arctan \frac{1}{\kappa_2}\sqrt{\frac{1-\kappa_2^2}{1-\kappa_1^2}},&\text{if }  \kappa_2\ne 0\\
\frac\pi2,&\text{if }\kappa_2=0
\end{cases}
\end{equation}
and
\begin{equation}\label{tana2}
\beta = \begin{cases}
\arctan (\kappa_1\kappa_2 \tan \alpha),&\text{if }  \kappa_2\ne 0\\
\frac{\kappa_1}{\sqrt{1-\kappa_1^2}},&\text{if }  \kappa_2= 0
\end{cases}.
\end{equation}

By Theorem \ref{technich0} (i),
  the  one-dimensional subspace
\begin{equation}\label{netral}
\cL_0
=\linspan\left \{ (\sin \alpha g_+^1-\sin \beta  g_-^1)
\oplus(\cos \alpha
g_+^2- \cos\beta g_-^2 )
\right \}
\end{equation}
is  a neutral subspace  of the quotient  space
$$\Dom ((\dot A_1\oplus\dot A_2)^*) / \Dom (\dot A_1\oplus\dot A_2).
$$

By Theorem \ref{technich0} (ii), the  restriction $\dot A_0$ of
the operator  $(\dot A_1\oplus \dot A_2)^*$
on the domain
\begin{equation}\label{vbnm}
\Dom(\dot A_0)=\Dom (\dot A_1)\oplus \Dom (\dot A_2)\dot +\cL_0
\end{equation}
is a symmetric operator with deficiency indices $(1,1)$.

Taking into account the relations (see \eqref{tana1}, \eqref{tana2})
$$\sin \beta =\kappa_1\sin\alpha \quad \text{and }\quad \cos \beta=\frac{1}{\kappa_2}\cos \alpha,\quad \kappa_2\ne 0,
$$ and
$$\sin \beta =\kappa_1 \quad \text{and }\quad \cos \beta=\sqrt{1-\kappa_1^2},\quad  \kappa_2= 0,$$
from \eqref{netral} one obtains that the subspace $\cL_0$ admits the representation
$$\cL_0=
\begin{cases}
\linspan\left \{ \sin \alpha \,\, (g_+^1-\kappa_1 g_-^1)
\oplus\cos \alpha \left (
g_+^2- \frac{1}{\kappa_2}g_-^2 \right )\right \},&\text{if }\kappa_2\ne 0,
\\
\linspan\left  \{  (g_+^1-\kappa_1 g_-^1)
\oplus
\left (-\sqrt{1-\kappa_1^2}g_-^2 \right )\right \},&  \text{if } \kappa_2=0.
\end{cases}
$$
It follows that
$$\cL_0\subset \dom (\widehat A_1\oplus (\widehat A_2)^*).
$$
From \eqref{vbnm} one concludes  that the symmetric operator $\dot A_0$  has the property
\begin{equation}\label{111}\dot A_0\subset \widehat A_1\oplus (\widehat A_2)^*.
\end{equation}

Clearly, for any $\theta\in [0,2\pi)$ the subspace
\begin{equation}
\cL_\theta=
\begin{cases}\label{sravni}
\linspan\left \{ e^{i\theta}\sin \alpha \,\, (g_+^1-\kappa_1 g_-^1)
\oplus\cos \alpha \left (
g_+^2- \frac{1}{\kappa_2}g_-^2 \right )\right \},&\text{if }\kappa_2\ne 0,
\\
\linspan\left  \{  e^{i\theta} (g_+^1-\kappa_1 g_-^1)
\oplus
\left (-\sqrt{1-\kappa_1^2}g_-^2 \right )\right \},&  \text{if }  \kappa_2=0.
\end{cases}
\end{equation}
is  also a neutral subspace  of the quotient  space
$$\Dom ((\dot A_1\oplus\dot A_2)^*) / \Dom (\dot A_1\oplus\dot A_2).
$$
Therefore, the symmetric operator $\dot  A_\theta$ defined as  the restrictions of $(\dot A_1\oplus\dot  A_2)^*$
on
$$\Dom(\dot  A_\theta)=\dom (\dot A_1\oplus \dot A_2)\dot + \cL_\theta, \quad \theta\in [0,2\pi),
$$has deficiency indices $(1,1) $ and
$$
 (\dot A_1\oplus \dot A_2)\subset \dot A_\theta \subset (\widehat A_1\oplus (\widehat A_2)^*)\subset  (\dot A_1\oplus \dot A_2)^*, \quad \theta\in
 [0,2\pi),
$$
proving the claim (i).

(ii). Introduce the elements
\begin{equation}\label{eqref1}
f^1=g_+^1-\kappa_1g_-^1\in \Dom (\widehat A_1)
\subset \cH_1
\end{equation}
and
\begin{equation}\label{eqref2}
f^2=g_+^2-\kappa_2^{-1}g_-^2 \in
\Dom ((\widehat A_2)^*)\subset \cH_2 \quad (\kappa_2\ne0).
\end{equation}
If $\kappa_2=0$, then  we take
\begin{equation}\label{eqref3}
f^2=-\sqrt{1-\kappa_1^2}\,g_-^2 \in
\Dom ((\widehat A_2)^*)\subset \cH_2.
\end{equation}

A simple computation shows that
\begin{equation}\label{fedor}
\Im (\widehat A_1 f^1,f^1)=(1-\kappa_1^2)>0\end{equation}
and that
\begin{equation}\label{fedor2}
\Im ((\widehat A_2)^* f^2,f^2)=
\begin{cases}
1-\kappa_2^{-2},&\text{if } \kappa_2\ne 0\\
\kappa_1^{2}-1,&\text{if } \kappa_2=0
\end{cases}.
\end{equation}
Hence,
$
\Im ((\widehat A_2)^* f^2,f^2)<0
$.
Therefore, if
$
f=af^1+bf^2$,  $a,b\in \bbC$,
then
$$
\Im((\dot A_1\oplus \dot A_2)^*f,f)=
\begin{cases}
|a|^2(1-\kappa_1^2)+|b|^2(1-\kappa_2^{-2}),& \kappa_2\ne0\\
|a|^2(1-\kappa_1^2)-|b|^2(1-\kappa_1^{2}),& \kappa_2=0
\end{cases}.
$$

This means that   a one-dimensional subspace
$$\cL \subset\linspan \{f^1,f^2\}\subset  \Dom   (\widehat A_1)\oplus \Dom((\widehat A_2)^*)$$
 is a neutral (Lagrangian) subspace for the symplectic form
$$\omega(h,g)=((\dot A_1\oplus \dot A_2)^*h,g)-(h,(\dot A_1\oplus \dot A_2)^*g), \quad h,g\in \Dom((\dot A_1\oplus \dot A_2)^*)
$$
if and only if $\cL$ admits the representation
\begin{equation}\label{cl}\cL=\linspan\{e^{i\theta}\sin \alpha f^1\oplus\cos \alpha f^2\} ,
\end{equation}
for some $\theta\in [0,2\pi)$ where
$$\tan \alpha =\frac{\kappa_2^{-2}-1}{1-\kappa_1^2}=\frac{1}{\kappa_2}\sqrt{\frac{1-\kappa_2^2}{1-\kappa_1^2}}$$
if $\kappa_2\ne 0$, and
\begin{equation}\label{cl1}
\cL=\linspan\{e^{i\theta}f^1\oplus f^2\},
\end{equation}
if $\kappa_2=0$.

 Taking into account \eqref{eqref1}--\eqref{eqref3} and comparing  \eqref{cl} and \eqref{cl1} with \eqref{sravni}, one concludes that
\begin{equation}\label{vso}
\cL=\cL_\theta.
\end{equation}

By hypothesis (ii), $\dot A$ is a closed symmetric operator with deficiency indices $(1,1)$  and
 $$
\dot A_1\oplus \dot A_2\subset \dot A\subset \widehat A_1\oplus (\widehat  A_2)^*.
 $$
Therefore, the subspace
$$
\Dom(\dot A)\cap\Dom( \widehat A_1\oplus (\widehat  A_2)^*)
$$
is a neutral subspace.  Hence,  by \eqref{vso},
$$
\Dom(\dot A)=\Dom(\dot A_1\oplus \dot A_2)\dot +\cL_\theta \quad\text{for some }\quad \theta\in[0,2\pi)
$$
which means that
$$\dot A=\dot A_\theta
$$
proving the claim (ii).

The proof is complete.
\end{proof}

Our next result,  on  the one hand, shows that given   a solution $\dot A$ of the extension problem with a constraint \eqref{EPC}, there exists a unique
operator coupling $\widehat A_1\uplus\widehat A_2$ of  $\widehat A_1$ and
$\widehat A_2$ such that
$$\dot A\subset \widehat A_1\uplus\widehat A_2.
$$
On the other hand, this result justifies that
  the functional
$$\widehat \kappa :\fD\to [0,1)$$
introduced in subsection \ref{ssylka}  is  multiplicative with respect to
the operator coupling operation.

\begin{theorem}[{\bf Multiplicativity of  the  extension parameter}]\label{nachalo}
Assume Hypothesis \ref{hyphyp}.
Suppose, in addition,  that  $\dot A$ is a solution of the extension problem with a constraint \eqref{EPC}.

 Then

 \begin{itemize}
 \item[(i)] there exists a unique operator coupling   $\widehat A=\widehat A_1\uplus \widehat A_2\in \fD(\cH_1\oplus \cH_2)$  such that
  $$
 \widehat A|_{\Dom (\widehat A)\cap\dom( \widehat A)^*)}=\dot A;
 $$

\item[(ii)]  for any operator coupling $\widehat A$ of $\widehat A_1$ and $\widehat A_2$, the multiplication rule
\begin{equation}\label{multkappa1}\widehat \kappa(\widehat A)=\widehat \kappa (\widehat A_1)\cdot \widehat \kappa(\widehat A_2)
\end{equation}
holds.
Here $\widehat \kappa(\cdot )$ stands for  the absolute value of the von Neumann parameter of a dissipative  operator.

\end{itemize}
\end{theorem}

\begin{proof} (i).
As in the proof of Lemma \ref{dlinnaia}, start with  a basis
$g_\pm^j \in \Ker ((\dot A_j)^*\mp iI)$, $\|g_\pm^j\|=1$, $j=1,2$, in the corresponding deficiency subspaces
such that
\begin{equation}\label{bubu}
g_+^j-\kappa_jg_-^j\in \Dom (\widehat A_j),
\quad j=1,2,
\end{equation}
where $\kappa_j$ stands for the absolute value of the von Neumann parameter of $\widehat A_j$,
$$\kappa_j=\widehat \kappa (\widehat A_j), \quad j=1,2.
$$

By Lemma \ref{dlinnaia}, the domain of $\dot A$ admits the representation
$$\dom (\dot A)=\dom(\dot A_1\oplus \dot A_2)\dot +\cL_\theta,
$$
where
\begin{equation}\label{sravnieche}\cL_\theta=
\begin{cases}
\linspan\left \{ e^{i\theta}\sin \alpha \,\, (g_+^1-\kappa_1 g_-^1)
\oplus\cos \alpha \left (
g_+^2- \frac{1}{\kappa_2}g_-^2 \right )\right \},&\text{if } \kappa_2\ne 0
\\
\linspan\left  \{  e^{i\theta} (g_+^1-\kappa_1 g_-^1)
\oplus
\left (-\sqrt{1-\kappa_1^2}g_-^2 \right )\right \},& \text{if }  \kappa_2=0
\end{cases}
\end{equation}
and
\begin{equation}\label{tana11}
\tan \alpha=\frac{1}{\kappa_2}\sqrt{\frac{1-\kappa_2^2}{1-\kappa_1^2}}, \quad \kappa_2\ne 0.
\end{equation}
 Without loss one may assume that $\theta=0$. Indeed, instead of taking the basis $g_\pm^1 \in\Ker ((\dot A_1)^*\mp iI)$, one can start with the basis
$ e^{i\theta}g_\pm^1 \in\Ker ((\dot A_1)^*\mp iI)$ without changing the von Neumann extension parameter $\kappa_1$ that characterizes
the domain of $ \widehat A_1$ (see eq. \eqref{bubu}).

Taking into account  the relations $$\sin \beta =\kappa_1\sin\alpha \quad \text{and }\quad \cos \beta=\frac{1}{\kappa_2}\cos \alpha,
\quad \text{if }  \kappa_2\ne 0,
$$ and
$$\sin \beta =\kappa_1 \quad \text{and }\quad \cos \beta=\sqrt{1-\kappa_1^2},\quad  \text{if }  \kappa_2= 0,$$
it is easy to see that
\begin{equation}\label{netral1}
\cL_0
=\linspan\left \{ (\sin \alpha g_+^1-\sin \beta  g_-^1)
\oplus(\cos \alpha
g_+^2- \cos\beta g_-^2 )
\right \}.
\end{equation}

In accordance with Theorem \ref{technich0}, introduce the maximal dissipative extension $\widehat A$ of $\dot A$
defined as the restriction of $(\dot A_1\oplus\dot A_2)^*$ on
$$\Dom (\widehat A )=\dom (\dot A)\dot +\linspan \left \{ G_+-\kappa_1\kappa_2 G_-\right \},
$$
where the deficiency elements $G_\pm$ of $\dot A$ are given by
\eqref{G+G-}. That is,
\begin{equation}\label{GGG}
G_+=  \cos \alpha \, g_+^1-\sin  \alpha \,g_+^2,
\end{equation}
$$
G_-= \cos  \, \beta  g_-^1- \sin \beta \,g^2_-.
$$

By construction,
\begin{equation}\label{222}\dot A=\widehat A|_{\dom (\widehat A_1)\cap \dom((\widehat A_2)^*)}.
\end{equation}
Clearly,
\begin{align*}
G_+-\kappa_1\kappa_2 G_-&=(\cos \alpha \, g_+^1 -\kappa_1\kappa_2   \cos  \, \beta  g_-^1)\oplus  (-\sin  \alpha \,g_+^2+\kappa_1\kappa_2 \sin \beta \,g^2_-)\\
&=\begin{cases}
\cos \alpha \, ( g_+^1 -\kappa_1  g_-^1)\oplus (- \sin \alpha) \, (g_+^2-k_1^2k_2g_-^2),&\text{if } \kappa_2\ne 0\\
 0\oplus (-g_+^2),&\text{if } \kappa_2= 0
\end{cases}.
\end{align*}
Therefore,
$$
\text{Proj}_{\cH_1}( G_+-\kappa_1\kappa_2 G_-)\in \dom (\widehat A_1),
$$
where   $\text{Proj}_{\cH_1}$ denotes the orthogonal projection of  $\cH_1\oplus \cH_2$ onto $\cH_1$.  Hence,
the subspace $\cH_1$ is invariant for the dissipative operator $\widehat A$ and\begin{equation}\label{333}
\widehat A|_{\cH_1\cap \Dom(\widehat A_1)}=\widehat A_1.
\end{equation}
Combining \eqref{111}, \eqref{222} and \eqref{333} shows that the dissipative extension $\widehat A$ is an operator coupling of $\widehat A_1$ and $\widehat A_2$ such that $\dot A\subset \widehat A_1\uplus
\widehat A_2
$,
which proves the existence part of the assertion.

To prove the uniqueness of the operator coupling $\widehat A$  extending $\dot A$ and satisfying the property \eqref{333},
one observes that since $\widehat  A\in \fD(\cH_1\oplus\cH_2)$, there exists  some $|\kappa|<1$ such that
$$
\Dom (\widehat A)=\dom (\dot A)\dot +\linspan \left \{ G_+-\kappa G_-\right \}.
$$
In particular,
\begin{equation}\label{means}
G_+-\kappa G_-\in \Dom (\widehat A).
\end{equation}

If $\kappa_2\ne 0$, from \eqref{333} it follows that
\eqref{means} holds if and and only if
$$
\text{Proj}_{\cH_1}(G_+-\kappa G_-)=\cos \alpha g_+^1-\kappa \cos \beta g_-^1=\cos \alpha \left (g_+^1-\frac{\kappa}{\kappa_2}  g_-^1\right )\in \dom (\widehat A_1)
$$
which is only possible if
$$
\frac{\kappa}{\kappa_2}=\kappa_1.
$$

If $\kappa_2=0$, and therefore in this case  $G_+=-g_+^2$ (see \eqref{GGG} with $\alpha=\frac\pi2$),
one computes
$$\text{Proj}_{\cH_1}(G_+-\kappa G_-)=-\kappa \cos\beta g_-^1 \in \Dom (\widehat A),
$$
and hence \eqref{333} and \eqref{means} hold if and only if $\kappa=\kappa_2=0$.
In particular, we have shown that in either case
\begin{equation}\label{prodprod}\kappa=\kappa_1\kappa_2.
\end{equation}

(ii).
By definition of the von Neumann parameter associated with a pair of operators, equality \eqref{prodprod} means that
$$
\kappa(\widehat A, A)=\kappa(\widehat A_1, A_1)\cdot \kappa(\widehat A_2, A_2),
$$
where $A$ and $A_j$, $j=1,2$, are self-adjoint reference extensions of $\dot A$ and $\dot A_j$, $j=1,2$, such that
$$G_+-G_-\in \Dom (A)
$$
and
$$
g_+^j-g_-^j\in \Dom (A_j),\quad j=1,2,
$$
which proves the remaining assertion \eqref{multkappa1}.

The proof is complete.
\end{proof}

\section{The  multiplication theorem}

Now, we are ready to state the central result of this  paper.

\begin{theorem}[{\bf The Mutiplication Theorem}]\label{opcoup}
Suppose   that $\widehat A=\widehat A_1\uplus \widehat A_2$ is an operator coupling  of two  maximal dissipative operators
  $\widehat A_k \in \fD(\cH_k)$, $k=1,2$. Denote by  $\dot A $, $\dot A_1$ and $\dot A_2$  the corresponding underlying symmetric operators with deficiency indices $(1,1)$, respectively.
That is,  $$\dot A=\widehat A|_{\Dom(\widehat A)\cap\Dom((\widehat A)^*)}
$$
and
$$
\dot A_k=\widehat A_k|_{\Dom(\widehat A_k)\cap\Dom((\widehat A_k)^*)}, \quad k=1,2.
$$

  Then there exist self-adjoint reference operators  $A$, $A_1$, and $A_2$,  extending   $\dot A$, $\dot A_1$ and $\dot A_2$, respectively,  such that
\begin{equation}\label{proizvas}S(\widehat A_1\uplus \widehat A_2,  A)=S(\widehat A_1,A_1)\cdot S(\widehat A_2, A_2).
\end{equation}
\end{theorem}

\begin{proof}
As in the proof of Theorem \ref{nachalo}, one can always find a basis
$$g_\pm^j \in \Ker ((\dot A_j)^*\mp iI),\quad \|g_\pm^j\|=1,\quad j=1,2, $$
such that
$$
g_+^j-\kappa_jg_-^j\in \Dom (\widehat A_j),
\quad \text{with} \quad  \kappa_j=\widehat \kappa (\widehat A_j),\quad j=1,2,
$$
and that
$$\dom (\dot A)=\dom (\dot A_1\oplus\dot A_2)\dot +\cL_0.
$$
Here
$$
\cL_0
=\linspan\left \{ (\sin \alpha g_+^1-\sin \beta  g_-^1)
\oplus(\cos \alpha
g_+^2- \cos\beta g_-^2 )\right \}
$$
and
\begin{equation}\label{tana111}
\alpha=\arctan \frac{1}{\kappa_2}\sqrt{\frac{1-\kappa_2^2}{1-\kappa_1^2}}
\quad \left (\alpha=\frac\pi2 \quad \text{if } \quad \kappa_2=0\right ),\end{equation}
 \begin{equation}\label{sinsin}
 \sin \beta =\kappa_1
 \begin{cases}
 \sin\alpha,&\text{if } \kappa_2\ne0\\
1,&\text{if } \kappa_2=0
 \end{cases},
\end{equation}
 \begin{equation}\label{coscos}
\cos \beta=
 \begin{cases}
\frac{1}{\kappa_2}\cos \alpha,&\text{if } \kappa_2\ne 0\\
 \sqrt{1-\kappa_1^2},&\text{if }  \kappa_2= 0
 \end{cases}.
 \end{equation}

By  Theorem \ref{technich0}, the deficiency elements $G_\pm$ of $\dot A$ are given by
\eqref{G+G-},
\begin{equation}\label{GGGG}
G_+=  \cos \alpha \, g_+^1-\sin  \alpha \,g_+^2,
\end{equation}
$$
G_-= \cos  \, \beta  g_-^1- \sin \beta \,g^2_-.
$$

Introducing   self-adjoint reference extensions $A$ and $A_j$, $j=1,2$,  of the symmetric operators $\dot A$ and $\dot A_j$, $j=1,2$, such that
$$G_+-G_-\in \Dom (A)\quad \text{
and}\quad
g_+^j-g_-^j\in \Dom (A_j),\quad j=1,2,
$$
one can apply  Theorem \ref{technich0} to conclude that  the  Liv\v{s}ic function of $\dot A$ relative to $A$ admits the representation
\begin{equation}\label{sysy11}
s(z)=s(\dot A, A)(z)=\frac{\cos \alpha \cos \beta s_1(z)-s_1(z)s_2(z)
+ \sin \alpha \sin \beta s_2(z)
}
{1   - (\sin \alpha \sin \beta s_1(z)+
\cos  \alpha \cos  \beta s_2(z) ) }.
\end{equation}
Here
$$s_k(z)=s(\dot A_k, A_k)
$$
are the Liv\v{s}ic functions associated with the pairs $ (\dot A_k, A_k)$, $ k=1,2$.

Denote the operator coupling $\widehat A_1\uplus\widehat A_2$ by $\widehat A$.
By Theorem \ref{nachalo},
\begin{equation}\label{nuii}
G_+-\kappa_1\kappa_2G_-\in \Dom (\widehat A).
\end{equation}
Therefore,  from \eqref{nuii} it follows that
the characteristic function $S(\widehat A, A)$ of the dissipative extension $\widehat A$ relative to the reference self-adjoint operator $A$
has the  form
$$
S(\widehat A, A)(z)=\frac{s(z)-\kappa_1\kappa_2}{
\kappa_1\kappa_2s(z)-1}.
$$

By Remark \eqref{kogdage} with $\kappa=\kappa_1\kappa_2$, one gets that
  $$
\frac{s(z)-\kappa_1\kappa_2}{
\kappa_1\kappa_2s(z)-1}=\frac{a_1 s_1(z)+a_2 s_2(z)-
s_1(z)s_2(z)
-\kappa_1\kappa_2}
{a_2 s_1(z)+a_1 s_2(z)-\kappa_1\kappa_2s_1(z)s_2(z)-
1},
$$
where
$$
a_1=\cos \alpha \cos \beta +\kappa_1\kappa_2\sin \alpha \sin \beta,
$$
$$
a_2=\sin \alpha \sin \beta+\kappa_1\kappa_2\cos \alpha \cos \beta.
$$
From the relations \eqref{tana111}, \eqref{sinsin} and \eqref{coscos} it  follows that $a_1=\kappa_2$  and $a_2=\kappa_1$ and hence
\begin{align}
\frac{s(z)-\kappa_1\kappa_2}{
\kappa_1\kappa_2s(z)-1}&=\frac{\kappa_2 s_1(z)+\kappa_1 s_2(z)-s_1(z)s_2(z)
-\kappa_1\kappa_2 }
{\kappa_1s_1(z)+\kappa_2s_2(z)-\kappa_1\kappa_2s_1(z)s_2(z)-
1}
\nonumber \\
&=
\frac{s_1(z)-\kappa_1}
{\kappa_1 s_1(z)-1}
\cdot \frac{s_2(z)-\kappa_2}
{\kappa_2 s_2(z)-1}.\label{nunu}
\end{align}
Thus,
\begin{equation}\label{prozv}
S(\widehat A, A)(z)= S(\widehat A_1, A_1)(z)\cdot S(\widehat A_2, A_2)(z),
\quad z\in \bbC_+.
\end{equation}

The proof is complete.

\end{proof}

The following example illustrates the Multiplication Theorem \ref{opcoup} for a differentiation operator on a finite interval.

\begin{example}\label{exam} For   a finite interval $\delta=[\alpha, \beta]$, denote by $\widehat D_\delta$
 the   first order differentiation  operator
in the Hilbert space $L^2(\delta)$
given by  the differential expression
 $$
\tau=-\frac{1}{i} \frac{d}{dx}
$$
on
$$
\Dom(\widehat D_\delta )=\left \{f\in W_2^1((\alpha, \beta )),\,\, f(\alpha)=0\right \}
.
$$
It is easy to see that if $\gamma\in (\alpha, \beta)$, and therefore $$ \delta=\delta_1\cup \delta_2,$$
 with $\delta_1=[\alpha,\gamma]$ and $\delta_2=[\gamma, \beta]$, then
\begin{equation}\label{ccoo}
 \widehat D_\delta= \widehat D_{\delta_1\cup\delta_2}=\widehat D_{\delta_1}\uplus \widehat D_{\delta_2},
\end{equation}
where  $\widehat D_{\delta_1}\uplus \widehat D_{\delta_2}$ stands for  the dissipative operator coupling of $\widehat D_{\delta_1}$ and $\widehat D_{\delta_2}$.

Indeed, by construction,  $\widehat D_\delta$ is a maximal dissipative extension of $\widehat D_{\delta_1}$ outgoing from the Hilbert space $\cH_1=L^2(\delta_1)$ to the Hilbert space
$\cH=\cH_1\oplus\cH_2=L^2(\delta)$, where $\cH_2=L^2(\delta_2)$. Moreover, since
$$
\Dom((\widehat D_\delta)^*)=\left \{f\in W_2^1((\alpha, \beta )),\,\, f(\beta)=0\right \},
$$
the restriction $\dot D_\delta$ of $\widehat D_\delta $ on
\begin{equation}\label{domu}
\Dom (\dot D_\delta)=\Dom(\widehat D_\delta)\cap\Dom(  (\widehat D_\delta)^*)
\end{equation}
is a symmetric operator with deficiency indices $(1,1)$ given by the  same differential expression  $\tau$ on
$$
\Dom (\dot D_\delta)=\left \{f\in W_2^1((\alpha, \beta )),\,\,f(\alpha)= f(\beta)=0\right \}.
$$
On the other hand,
$$
\Dom((\widehat D_{\delta_2})^*)=\left \{f\in W_2^1((\gamma, \beta )),\,\, f(\beta)=0\right \}.
$$
Therefore,
\begin{equation}\label{bkluch}\dot D_\delta\subset \widehat D_{\delta_1}\oplus (\widehat D_{\delta_2})^*.
\end{equation}
Combining \eqref{domu} and \eqref{bkluch} shows that $\widehat D_\delta$ coincides with  the dissipative operator coupling of $\widehat D_{\delta_1}$ and $\widehat D_{\delta_2}$. That is,
\eqref{ccoo} holds.

By Lemma \ref{app} (see Appendix A), the Liv\v{s}ic function associated with the pair $(\widehat D_\delta,  D_\delta)$  is of the form
$$S(\widehat D_\delta,  D_\delta)(z)=\exp ( i |\delta|z), \quad z\in \bbC_+,
$$
where $|\cdot|$ stands for Lebesgue measure of a Borel set and $ D_\delta$ is the self-adjoint reference   differentiation operator with antiperiodic boundary conditions defined on
$$
\Dom (D_\delta)=\left \{f\in W_2^1((\alpha, \beta )),\,\,f(\alpha)= -f(\beta)\right \}.
$$

Therefore, taking into account that
$$
\exp ( i |\delta|z)=\exp\left (i(|\delta_1|+|\delta_2|)z\right )=\exp ( i |\delta_1|z)\cdot \exp ( i |\delta_2|z),
$$
one obtains that
$$
S(\widehat D_\delta, D_\delta)(z)=S(\widehat D_{\delta_1}, D_{\delta_1})(z)\cdot S(\widehat D_{\delta_2},  D_{\delta_2})(z),
$$
which illustrates  the statement of Theorem \ref{opcoup}.

\end{example}

 We conclude this section by the   following  purely analytic result.

\begin{theorem} \label{mainan}Let $\fM$, $\fC$, and $\fS$ be the function classes of Weyl-Titchmarsh, Liv\v{s}ic, and characteristic functions, respectively.
Then,
\begin{itemize}
\item[(i)] The class $\fM$  is a convex set with respect to addition;
\item[(ii)] The class $\fS$  is closed under multiplication,
$$\fS\cdot \fS\subset \fS;
$$
\item[(iii)] The subclass  $\fC\subset \fS$   is  a (double sided) ideal  under multiplication in the sense that
$$
\fC\cdot \fS=\fS\cdot \fC\subset \fC;
$$

\item[(iv)] The class $\fC$   is closed under multiplication:
$$\fC\cdot\fC\subset \fC.$$

\end{itemize}
\end{theorem}
\begin{proof} One notices that (i) is a corollary of Theorem \ref{additionth},
(ii) follows from Theorem \ref{opcoup}, and
(iv) follows from (iii). Therefore, it remains to prove (iii).

(iii). Suppose that  $S_1\in \fC$ and $S_2\in \fS$. Since $\fC\subset \fS$ and $\dot \fD=\ker \widehat \kappa$,
 $S_1$ is the characteristic function of a dissipative operator $\widehat  A_1$ from  $\dot \fD$ (see \eqref{dotD})
relative to some self-adjoint reference operator $A_1$. Since $S_2\in \fS$,  the function  $S_2$ is the characteristic function of a dissipative operator
 $\widehat A_2\in \fD$
relative to some self-adjoint reference operator $A_2$. By Theorem \ref{opcoup}, the product $S_1\cdot S_2$ is
 the characteristic function of an operator coupling $\widehat A_1\uplus \widehat  A_2$ relative to  an appropriate reference self-adjoint operator.
Since $S_1\in \fC$, and therefore $\widehat \kappa (\widehat A_1)=0$, from  Theorem \ref{opcoup} follows that
 $\widehat \kappa(\widehat A_1\uplus\widehat A_2)=0$ and hence the product $S_1\cdot S_2$ belongs to the class $\fC$.
\end{proof}

\begin{remark} Recall that the subclass $\dot \fD$ of $\fD$ has been  defined as the set of all dissipative operators form $\fD$ with zero value of the corresponding von Neumann parameter (see   \eqref{dotD}).  To express this in a different way, the characteristic functions for the operators from $\dot \fD$ are exactly those
 that  belong the class $\fC$.

Having this in mind, a non-commutative version of the ``absorption principle'' (iii) can be formulated as follows.

Suppose that $\widehat A\in \dot \fD(\cH_1)\subset \fD(\cH_1)$ and  $\widehat B\in \fD(\cH_2)$. Then
$$
\widehat A\uplus \widehat B\in \dot \fD(\cH_1\oplus \cH_2)
$$
and
$$
\widehat B\uplus \widehat A\in \dot \fD(\cH_2\oplus \cH_1).
$$
\end{remark}

\begin{appendix}
\section{The differentiation on a finite interval}

In this Appendix we collect some known results, see, e.g., \cite{AkG}, regarding  the maximal and minimal  differentiation operators on a finite interval.
Here we present them in a version adapted  to the notation of the current paper.

\begin{lemma}\label{app}
 Let $\widehat D$ be the   first order differentiation  operator
in the Hilbert space $L^2(0, \ell)$
given by  the differential expression
 $$
\tau=-\frac{1}{i} \frac{d}{dx}
$$
on
$$
\Dom(\widehat D )=\left \{f\in W_2^1((0, \ell)),\,\, f(0)=0\right \}
$$ and $D$ the self-adjoint realization of $\tau$ on
$$
\Dom( D )=\left \{f\in W_2^1((0, \ell)),\,\, f(\ell)=-f(0)\right \}.
$$
Then \begin{itemize}
\item[(i)]
the restriction $\dot D$ of the operator $\widehat D$ on
\begin{align}
\Dom (\dot D&)=\dom (\widehat D)\cap \dom(\widehat D^*) \label{dommm}
\end{align}
is a  symmetric operator with deficiency indices $(1, 1)$;

\item[(ii)]
the Liv\v{s}ic  function $s=s(\dot D, D)$
of the symmetric  operator $\dot D$ relative to the self-adjoint reference operator $D$
is of the form
\begin{equation}\label{chrff3}
  s(z)=
\frac{e^{i\ell z}-e^{-\ell}}{e^{-\ell}e^{i\ell z}-1};
\end{equation}
\item[(iii)] the von Neumann parameter $\kappa(\dot D, \widehat D)$ associated with the pair  $(\dot D, \widehat D) $ is given by
$$
\kappa(\dot D, \widehat D)=e^{-\ell};
$$

\item[(iv)] the characteristic function $S=S(\widehat D,D)$ of the dissipative operator $\widehat D$ relative to $D$ is an inner singular function given by
$$
S(z)=e^{i\ell z}, \quad z\in \bbC_+.$$
\end{itemize}
\end{lemma}
\begin{proof} It is straightforward to conclude that
$$
\Dom(\widehat D^*)=\left \{f\in W_2^1((0, \ell)),\,\, f(\ell)=0\right \}
$$
and therefore
\begin{align}
\Dom (\dot D)=\dom (\widehat D)\cap \dom(\widehat D^*)
=\left \{f\in W_2^1((0,\ell)),\,\,f(0)= f(\ell)=0\right \}\nonumber.
\end{align}

Clearly,
$
\Ker ((\dot D)^*-z I)=\linspan\{ g_z\},
$
where
$$
g_z(x)=e^{-izx}, \quad x\in [0, \ell], \quad z\in \bbC,
$$
which proves (i).

To compute the Liv\v{s}ic function, one observes that
$
\Ker ((\dot D)^*\mp iI)=\linspan\{ g_\pm\},
$
where
$$
g_+(x)=\frac{\sqrt{2}}{\sqrt{e^{2\ell}-1}}e^{ x}
\quad \text{ and } \quad
g_-(x)=\frac{\sqrt{2}}{\sqrt{1-e^{-2\ell}}}e^{ -x}, \quad x\in [0, \ell],
$$
Obviously,
$\|g_\pm \|=1.
$

Since
$$g_+(0)-g_-(0)=\frac{\sqrt{2}}{\sqrt{e^{2\ell}-1}}-\frac{\sqrt{2}}{\sqrt{1-e^{-2\ell}}}=\frac{\sqrt{2}}{\sqrt{e^{2\ell}-1}}\left (1-e^{\ell}\right )
$$
and
$$\quad\quad\quad\,\, \,\,g_+(\ell)-g_-(\ell)=\frac{\sqrt{2}}{\sqrt{e^{2\ell}-1}}e^\ell-\frac{\sqrt{2}}{\sqrt{1-e^{-2\ell}}}e^{-\ell}=-\frac{\sqrt{2}}{\sqrt{e^{2\ell}-1}}\left (1-e^{\ell}\right ),
$$
one observes that
$
g_+(0)-g_-(0)=-(g_+(\ell)-g_-(\ell))
$ which proves that
\begin{equation}\label{propp}
g_+-g_-\in \Dom(D).
\end{equation}

Now, since \eqref{propp} holds, in accordance with definition
the Liv\v{s}ic  function $s=s(\dot D, D)$
of the symmetric  operator $\dot D$ relative to the self-adjoint reference operator $D$ can be evaluated as
\begin{align*}
s(z)&=\frac{z-i}{z+i}\cdot \frac{(g_z,g_-)}{(g_z,g_+)}=\sqrt{\frac{e^{2\ell}-1}{1-e^{-2\ell}}}\cdot
\frac{z-i}{z+i}\cdot\frac{\int\limits_0^\ell e^{(-iz-1)x}dx}{\int\limits_0^\ell e^{(-iz+1)x}dx}
\\&
=\sqrt{\frac{e^{2\ell}-1}{1-e^{-2\ell}}}\cdot\frac{e^{(-iz-1)\ell}-1}{e^{(-iz+1)\ell}-1}=
e^{\ell} \frac{e^{(-iz-1)\ell}-1}{e^{(-iz+1)\ell}-1}
\\ &=\frac{e^{-iz\ell}-e^\ell}{e^\ell e^{-iz\ell}-1},\quad z\in\bbC_+.
\end{align*}
Thus,
$$
  s(z)=
\frac{e^{i\ell z}-e^{-\ell}}{e^{-\ell}e^{i\ell z}-1},\quad z\in \bbC_+,
$$
which proves  the representation (ii).

Next, since
$
g_+(0)=e^{-\ell}g_-(0),
$
one also obtains  that
$$
g_+-e^{-\ell} g_-\in \Dom (\widehat D)
$$
which proves the assertion  (iii) taking into account \eqref{propp}.

 Finally,  one concludes that
the characteristic function $S=S(\widehat D, D)$ of the dissipative operator  $\widehat D$ relative to the self-adjoint reference operator $D$
is  given by
\begin{equation}\label{kapling3}
S(z)=\frac{s(z)-e^{-\ell}}{e^{-\ell}s(z)-1}=\frac{\frac{e^{i\ell z}-e^{-\ell}}{e^{-\ell}e^{i\ell z}-1}-e^{-\ell}}{e^{-\ell} \frac{e^{i\ell z}-e^{-\ell}}{e^{-\ell}e^{i\ell z}-1}-1}
=
e^{i\ell z}, \quad z\in \bbC_+,
\end{equation}
which proves (iv).

The proof is complete.
\end{proof}

\end{appendix}


\begin{thebibliography}{99}

\bibitem{ABT}
Yu.~Arlinskii, S.~Belyi, E.~Tsekanovskii,
 Conservative realizations of Herglotz-Nevanlinna functions,
 Operator Theory: Adv. and Appl.,
{\bf 217}, Birkh\"auser/Springer Basel AG, Basel,
 2011.

\bibitem{AkG}
N.~I.~Akhiezer, I.~M.~Glazman,
Theory of Linear Operators in
 Hilbert Space, Dover, New York, 1993.


\bibitem{AWT}
A.~ Aleman, R.~T.~W.~Martin, W.~T. ~Ross,
 {\it On a theorem of Liv\v{s}ic},
  J. Funct. Anal., {\bf 264}, 999--1048 (2013).


\bibitem{Br}
M.~S.~Brodskii,
Triangular and Jordan representations of linear operators,
Translations of Mathematical Monographs,
 Vol. 32.
Amer. Math. Soc., Providence, R.I., 1971.

\bibitem {Brod}
M.~S.~ Brodskii,
\textit{Unitary operator colligations and their characteristic functions},
Uspechi Mat. Nauk {\bf 33}, no. 4(202), 142--168 (1978), (Russian).
English transl.:
Russ. Math. Surveys, {\bf 33} (4), 159--191 (1978).

\bibitem{BL58}
M.~S.~Brodskii, M.~S.~Liv\v{s}ic,
\textit{Spectral analysis of
non-self-adjoint operators and intermediate systems},
Uspehi Mat. Nauk (N.S.), {\bf 13},  no. 1(79), 3--85 (1958), (Russian).
English transl.:  Amer. Math. Soc. Transl., (2), {\bf 13}, 265--346
(1960).


\bibitem{D}
 W.~F.~Donoghue,
\textit{On perturbation of spectra},
Commun. Pure and Appl. Math.,
{\bf 18}, 559--579 (1965).

 \bibitem{FKE}
F.~Gesztesy, K.~A.~Makarov, E.~Tsekanovskii,
\textit{An addendum to Krein's formula},
J. Math. Anal. Appl.,
{\bf 222}, 594--606 (1998).

\bibitem{GT}
F.~Gesztesy, E.~Tsekanovskii,
\textit{On Matrix-Valued Herglotz Functions},
Math. Nachr., {\bf 218}, 61--138 (2000).



\bibitem{L}
M.~S.~Liv\v{s}ic,
\textit{On a class of linear operators in Hilbert space},
Mat. Sbornik , (2), {\bf 19}, 239--262 (1946),
(Russian).
English transl.:  Amer. Math. Soc. Transl., (2), {\bf 13}, 61--83
(1960).



\bibitem{L1}
M.~S.~Liv\v{s}ic,
\textit{On spectral decomposition of linear non-self-adjoint operators},
 Mat. Sbornik , (76), {\bf 34}, 145--198 (1954),
(Russian).
 English transl.:  Amer. Math. Soc. Transl. (2), {\bf 5}, 67--114 (1957).

\bibitem{L2}
M.~S.~Liv\v{s}ic,
Operators, Oscillations, Waves (Open systems),
 Nauka, Moscow, 1966, (Russian).
 English transl.:  Transl. Math. Monograph,  {\bf 34}, AMS, Providence  R.~ I., 1973.


\bibitem{LYa}
M.~S.~Liv\v{s}ic, A.~A.~Yantsevich,
Operator colligations in Hilbert  spaces,
Winston, 1979.


\bibitem{LP}
M.~S.~Liv\v{s}ic, V.~P.~Potapov,
\textit{A theorem on the multiplication of characteristic matrix-functions},
Dokl. Acad. Nauk SSSR, {\bf  72}, 625--628 (1950), (Russian).


\bibitem{MT}
K.~ A.~ Makarov, E.~Tsekanovskii,
\textit{ On the Weyl Titchmarsh and Liv\v{s}ic functions},
Proceedings of Symposia in Pure Mathematics, Amer. Math. Soc., {\bf 87}, 291--313 (2013).


\bibitem{NF}
 B.~Sz.-Nagy, C.~Foias,
Harmonic analysis of operators on Hilbert space,
North-Holland, Amsterdam, 1970.


\end{thebibliography}
\end{document}